\documentclass[12pt,reqno]{amsart}
\usepackage{amsmath}
\usepackage{amsfonts}
\usepackage{amssymb}
\usepackage{amsthm}
\usepackage[alphabetic]{amsrefs} 
\usepackage{algorithm,setspace}
\usepackage{algpseudocode}
\usepackage{booktabs}
\usepackage{mathrsfs}
\usepackage{mdframed}
\usepackage{caption}
\usepackage[most]{tcolorbox}
\usepackage[makeroom]{cancel}
\usepackage{array}
\usepackage[shortlabels]{enumitem}
\usepackage[hidelinks]{hyperref}
\hypersetup{
    colorlinks=true,
    linkcolor=magenta,      
    citecolor=magenta,     
    filecolor=magenta,
    urlcolor=magenta,
}
\usepackage[noabbrev,nosort,capitalise]{cleveref}

\usepackage[margin=1in,headsep=30pt]{geometry}
\pdfpagewidth=\paperwidth
\usepackage{xcolor}
\usepackage{color,soul}
\usepackage{tikz}
\usepackage{tikz-cd}
\usepackage{tkz-euclide}
\usepackage[latin1]{inputenc}
\usepackage{fancyhdr}
\usepackage{stmaryrd}
\usetikzlibrary{
  graphs,
  graphs.standard
}
\usepackage{titlesec}

\titleformat{\section}
  {\normalfont\Large\bfseries}{\thesection }{1em}{}
\titleformat{\subsection}
  {\normalfont\large\bfseries}{\thesubsection.}{1em}{}
\titleformat{\subsubsection}
  {\normalfont\normalsize\itshape}{\thesubsubsection.}{1em}{}
\newcolumntype{C}{>{$}c<{$}}

\theoremstyle{plain}
\newtheorem{thm}{Theorem}
\newtheorem{cor}[thm]{Corollary}

\newtheorem{prop}[thm]{Proposition}

\newtheorem{defn}[thm]{Definition}

\newtheorem{rem}[thm]{Remark}
\newtheorem*{ex}{Example}

\numberwithin{equation}{section}

\fancypagestyle{Euler}{%
  \fancyhf{} 
  
  \fancyhead[LO]{\sl Grothendieck Groups}
  \fancyhead[RE]{Shihan Kanungo}
  \fancyhead[LE,RO]{\thepage}
}%
\DeclareMathOperator{\id}{id}
\DeclareMathOperator{\Sym}{Sym}

\newtheorem{theorem}{Theorem}[section]
\newtheorem{lemma}{Lemma}

\begin{document}

\newpage
\thispagestyle{empty}

\newpage 
\setlength{\parskip}{5pt}
\thispagestyle{empty}
\setcounter{page}{1}

\begin{center}
    {\LARGE \bf The Grothendieck Group and K-Theory} 
    \vskip 5pt
    {\normalsize Grothendieck groups linearize nonlinear classification}
    
    \vspace{0.3in}
    {\large  Shihan Kanungo} 
    
    \vspace{0.05in}
    {\small San Jos\'e State University}
\end{center}
\date{\today}

\vspace{0.2in}
\begin{center}
    {\bf Abstract}
\end{center}
\begin{quote}
    {\footnotesize In this expository paper, we develop the basic ideas underlying Grothendieck groups and to illustrate their appearance across algebra, topology, representation theory, and homological algebra. Motivated by the universal construction associated to a commutative monoid, we define the Grothendieck groups abelian categories and rings. Along the way we study several fundamental examples, including Euler characteristics, projective modules, and representation rings. We conclude with a discussion of $K$-theory and its applications, indicating how the elementary construction of $K_0$ serves as the first layer of a much richer homotopical theory.}
\end{quote} 

\setcounter{section}{0}

\vspace{5mm}
\section{Introduction}

One of the central ambitions of mathematics is classification. Given a class of mathematical objects, one seeks a collection of invariants that completely determines them up to isomorphism. In favorable situations this program succeeds spectacularly. Finite-dimensional vector spaces over a field are classified by dimension, finitely generated abelian groups admit a complete structure theorem, and finitely generated modules over a principal ideal domain decompose into cyclic pieces. These results share a common feature: their structure is sufficiently rigid that one can reduce classification to manageable algebraic data.

Beyond these classical examples, however, classification problems rapidly become intractable. Modules over a general ring resist any reasonable classification theory; even for comparatively simple algebras, the category of modules may be ``wild.'' Likewise, vector bundles over a topological space, coherent sheaves on an algebraic variety, or representations of an infinite group typically form enormous and poorly understood collections. The sheer abundance of such objects makes any direct attempt at classification unrealistic.

The failure of complete classification forces a change in perspective. Rather than attempting to distinguish objects individually, one instead seeks coarser invariants capturing structural information while remaining computable. In many contexts the most natural structure available is addition. Vector spaces admit direct sums, modules admit direct sums, vector bundles admit Whitney sums, and representations admit direct sums. Passing from objects to their isomorphism classes therefore produces a commutative monoid.

Grothendieck's fundamental insight was that such additive structures can be systematically ``linearized.'' Given a commutative monoid $M$, one constructs an abelian group $G(M)$, now called the \emph{Grothendieck group} of $M$, characterized by the universal property that every monoid homomorphism from $M$ to an abelian group factors uniquely through $G(M)$. In the simplest case, the passage
\[
\mathbb N \rightsquigarrow \mathbb Z
\]
is precisely the Grothendieck completion of the additive monoid of natural numbers.

The guiding philosophy of this paper is therefore the following:
\begin{quote}
\textit{Grothendieck groups linearize nonlinear classification problems.}
\end{quote}

\medskip
This principle manifests itself in a remarkable range of mathematical settings. In algebraic $K$-theory, Grothendieck groups provide the first approximation to the structure of projective modules and vector bundles. In representation theory, the representation ring
\[
R(G)=K_0(\mathrm{Rep}(G))
\]
encodes representations through formal additive combinations of irreducible objects. In topology, Grothendieck's construction applied to vector bundles yields topological $K$-theory, one of the most important generalized cohomology theories of the twentieth century. Through the Atiyah--Singer index theorem, $K$-theory connects analysis, geometry, and topology in profound ways.

The influence of Grothendieck groups extends equally deeply into algebraic geometry. The Grothendieck--Riemann--Roch theorem relates algebraic $K$-theory to intersection theory and characteristic classes, revealing that additive invariants of sheaves interact subtly with geometric pushforwards. In homological algebra and derived geometry, Grothendieck groups arise naturally from triangulated and derived categories, where Euler characteristics become manifestations of additivity relations in $K_0$.

The persistence of the same formal mechanism across these disparate domains indicates how fundamental Grothendieck's insight was.
Although the definitions are simple, the scope of the theory is vast. Because of this, Grothendieck groups occupy a unique position in modern mathematics: they are simultaneously bookkeeping devices, universal additive invariants, and shadows of deep geometric and homotopical structures. Their enduring importance lies in the fact that they provide a common algebraic language in contexts where direct classification is impossible.

\section{Defining the Grothendieck Group}

The simplest instance of a Grothendieck group arises from a commutative monoid. 

Given a commutative monoid $M$, its \emph{group completion} is the universal abelian group $G(M)$ equipped with a monoid homomorphism
\[
M \longrightarrow G(M).
\]
Equivalently, $G(M)$ may be described concretely as the abelian group generated by formal symbols $[m]$, one for each element $m\in M$, subject to the relations
\[
[a+b]=[a]+[b].
\]

In this way, the additive structure of the monoid is promoted to a genuine abelian group structure.

Several familiar constructions fit naturally into this framework. The passage
\[
\mathbb N \rightsquigarrow \mathbb Z
\]
is simply the group completion of the additive monoid of natural numbers. Likewise,
\[
(\mathbb Z_{\ne 0},\times)\rightsquigarrow (\mathbb Q^\times,\times)
\]
is the multiplicative analogue obtained by adjoining formal inverses to nonzero integers.

The true power of the construction appears when the monoid $M$ arises not merely from a set with an operation, but from the isomorphism classes of objects in a category.

We encountered a very special case of this phenomenon in \S VI.3.4, where we studied the category
\[
k\textsf{-Vect}^f
\]
of finite-dimensional vector spaces over a field $k$. 
There, the resulting Grothendieck group recovered precisely the information carried by dimension. In retrospect, dimension may be viewed as the universal additive invariant of finite-dimensional vector spaces.

The key observation is that the construction does not fundamentally depend on vector spaces themselves, but only on the existence of exact sequences. This leads naturally to the general notion of a Grothendieck group for an abelian or exact category.

An \emph{exact category} is, roughly speaking, an additive category equipped with a distinguished class of sequences behaving like short exact sequences. In many situations these are the ordinary short exact sequences of an abelian category, though other choices are possible. For instance, the class of split exact sequences forms an important example of an exact structure.

For the remainder of this section, however, we restrict ourselves to the classical setting of abelian categories equipped with their usual short exact sequences.

\begin{defn}
The \emph{Grothendieck group} $K_0(\mathcal C)$ of an abelian category $\mathcal{C}$ is the abelian group generated by symbols $[A]$, one for each isomorphism class of objects $A\in \mathcal C$, subject to the relations
\[
[B]=[A]+[C]
\]
for every short exact sequence
\[
0\longrightarrow A\longrightarrow B\longrightarrow C\longrightarrow 0
\]
in $\mathcal C$.
\end{defn}

Hueristically, the Grothendieck group ``forgets'' the internal extension data of the category while retaining its additive structure.

The guiding principle behind this construction is often described by the modern term \emph{decategorification}: one replaces a complicated category by a more tractable algebraic object, typically an abelian group or ring. Much of the subtlety of the original category is discarded, but enough information remains to capture important structural invariants.

We now establish several basic properties of Grothendieck groups.

\begin{lemma}\label{lem: direct sum}
In $K_0(\mathcal C)$, one has
\[
[A\oplus B]=[A]+[B].
\]
\end{lemma}

\begin{proof}
The direct sum fits into a short exact sequence
\[
0\longrightarrow A\longrightarrow A\oplus B\longrightarrow B\longrightarrow 0,
\]
so the defining relations of $K_0(\mathcal C)$ immediately give
\[
[A\oplus B]=[A]+[B].
\qedhere
\]
\end{proof}

Thus the addition law in the Grothendieck group extends the direct-sum operation of the category itself.
The next lemma shows that every element of $K_0(\mathcal C)$ may be represented as a formal difference of classes of objects.

\begin{lemma}\label{lem: difference}
Every element of $K_0(\mathcal C)$ can be written in the form
\[
[B]-[C]
\]
for some objects $B,C\in\mathcal C$.
\end{lemma}

\begin{proof}
By definition, an arbitrary element of $K_0(\mathcal C)$ has the form
\[
[B_1]+\cdots+[B_k]-[C_1]-\cdots-[C_\ell].
\]
Using the previous lemma repeatedly, this expression becomes
\[
[B_1\oplus\cdots\oplus B_k]
-
[C_1\oplus\cdots\oplus C_\ell],
\]
which proves the claim.
\end{proof}

As with many constructions in algebra, the true significance of the Grothendieck group is captured by its universal property.

\begin{prop}\label{prop: uprop of K0}
Let $G$ be an abelian group, and let
\[
\delta:\mathcal C\to G
\]
be a function satisfying the following conditions:
\begin{enumerate}[label=$(\arabic*)$]
    \item $\delta(A)=\delta(A')$ whenever $A\cong A'$,
    \item for every short exact sequence
    \[
    0\longrightarrow A\longrightarrow B\longrightarrow C\longrightarrow 0,
    \]
    one has
    \[
    \delta(B)=\delta(A)+\delta(C).
    \]
\end{enumerate}
Then there exists a unique group homomorphism
\[
\widetilde{\delta}:K_0(\mathcal C)\to G
\]
such that
\[
\widetilde{\delta}([A])=\delta(A)
\]
for every object $A\in\mathcal C$.
\end{prop}

\begin{proof}
Let $F(\mathcal C)$ denote the free abelian group generated by isomorphism classes of objects of $\mathcal C$. Since $\delta$ is invariant under isomorphism, it extends uniquely to a homomorphism
\[
F(\mathcal C)\to G.
\]

Now let $W\subset F(\mathcal C)$ be the subgroup generated by elements of the form
\[
[B]-[A]-[C]
\]
coming from short exact sequences
\[
0\to A\to B\to C\to 0.
\]
Condition (2) implies that every such generator lies in the kernel of the extended map, so
\[
W\subseteq \ker(\delta).
\]

Consequently, $\delta$ factors uniquely through the quotient
\[
K_0(\mathcal C)=F(\mathcal C)/W,
\]
yielding the desired homomorphism
\[
\widetilde{\delta}:K_0(\mathcal C)\to G.
\qedhere
\]
\end{proof}

This universal property explains why Grothendieck groups arise so naturally throughout mathematics: any additive invariant factors canonically through $K_0$. In this sense, $K_0(\mathcal C)$ is the universal additive invariant of the category $\mathcal C$.

We now turn to some concrete computations illustrating how Grothendieck groups recover familiar algebraic invariants. Let
\[
R\textsf{-Free}^{fg}
\]
denote the category of finitely generated free $R$-modules over an integral domain $R$. When $R=k$ is a field, this is precisely the category $k\textsf{-Vect}^f$.

\begin{prop}\label{prop: K0 of R-Free}
Let $R$ be an integral domain. Then
\[
K_0(R\textsf{\emph{-Free}}^{fg})\cong \mathbb Z.
\]
\end{prop}

Thus, for finitely generated free modules, the Grothendieck group recovers nothing more and nothing less than rank. Before proving the proposition, we first need to establish the additivity of rank in short exact sequences.

\begin{lemma}\label{lem: additivity of rank for free}
Let $R$ be an integral domain, and let
\[
0\longrightarrow U\longrightarrow V\longrightarrow W\longrightarrow 0
\]
be a short exact sequence of finitely generated free $R$-modules. Then
\[
\mathrm{rk}(V)=\mathrm{rk}(U)+\mathrm{rk}(W).
\]
\end{lemma}

\begin{proof}
Let $\alpha:V\to W$ denote the quotient map. Since the sequence is exact, we may identify
\[
U=\ker(\alpha)
\qquad\text{and}\qquad
W=\operatorname{im}(\alpha).
\]

Because $U$, $V$, and $W$ are free of finite rank, we may write
\[
U\cong R^{\oplus \mathrm{rk}(U)}, \qquad
V\cong R^{\oplus \mathrm{rk}(V)}, \qquad
W\cong R^{\oplus \mathrm{rk}(W)}.
\]

To reduce the problem to linear algebra, let $K$ denote the field of fractions of $R$. Extending scalars from $R$ to $K$, define
\[
\overline U=K^{\oplus \mathrm{rk}(U)}, \qquad
\overline V=K^{\oplus \mathrm{rk}(V)}, \qquad
\overline W=K^{\oplus \mathrm{rk}(W)}.
\]

The map $\alpha$ extends uniquely to a $K$-linear map
\[
\overline\alpha:\overline V\to \overline W.
\]

Since every element of $\overline W$ is a scalar multiple of an element of $W$, and since $\operatorname{im}(\alpha)=W$, it follows that
\[
\operatorname{im}(\overline\alpha)=\overline W.
\]
Similarly,
\[
\ker(\overline\alpha)=\overline U.
\]

We may therefore apply the rank--nullity theorem for vector spaces (Proposition VI.3.11) to obtain
\[
\dim(\overline V)
=
\dim(\overline U)
+
\dim(\overline W).
\]

Finally, by construction,
\[
\dim(\overline U)=\mathrm{rk}(U),
\qquad
\dim(\overline V)=\mathrm{rk}(V),
\qquad
\dim(\overline W)=\mathrm{rk}(W),
\]
and the result follows.
\end{proof}

We can now prove the proposition.

\begin{proof}[Proof of \cref{prop: K0 of R-Free}]
Define a function
\[
\delta:R\textsf{-Free}^{fg}\to \mathbb Z
\]
by
\[
\delta(V)=\mathrm{rk}(V).
\]

By the previous lemma, rank is additive on short exact sequences. Hence, by the universal property of the Grothendieck group, $\delta$ induces a homomorphism
\[
\widetilde\delta:
K_0(R\textsf{-Free}^{fg})
\longrightarrow
\mathbb Z.
\]

Since $\mathrm{rk}(R)=1$, the class $[R]$ maps to $1$, so $\widetilde\delta$ is surjective.

To prove injectivity, let
\[
x\in K_0(R\textsf{-Free}^{fg})
\]
lie in the kernel of $\widetilde\delta$. By Lemma 1.3, we may write
\[
x=[B]-[C]
\]
for some finitely generated free $R$-modules $B$ and $C$. Since $\widetilde\delta(x)=0$, we have
\[
\mathrm{rk}(B)=\mathrm{rk}(C).
\]

Now integral domains satisfy the invariant basis number (IBN) property, so free modules of the same finite rank are isomorphic. Thus
\[
B\cong C,
\]
which implies $[B]-[C]=0$ in $K_0(R\textsf{-Free}^{fg})$.
It follows that $\widetilde\delta$ is injective, and hence an isomorphism.
\end{proof}

The proposition shows that $K_0$ extracts precisely the information of rank from the category of finitely generated free modules. In this setting, rank completely determines the isomorphism class of an object, so the Grothendieck group loses essentially no information.

\begin{ex}\label{ex: countable rank}
Suppose instead that we consider the category
\[
R\textsf{-Free}^c
\]
of countably generated free $R$-modules. Then the corresponding Grothendieck group is trivial.
Indeed, if $A$ is any countable set, then
\[
R^{\oplus A}\oplus R^{\oplus \mathbb Z}
\cong
R^{\oplus \mathbb Z}.
\]
Passing to $K_0$, we obtain
\[
[R^{\oplus A}]
+
[R^{\oplus \mathbb Z}]
=
[R^{\oplus \mathbb Z}],
\]
and hence
\[
[R^{\oplus A}]=0.
\]

Since every object of $R\textsf{-Free}^c$ is isomorphic to some $R^{\oplus A}$, it follows that
\[
K_0(R\textsf{-Free}^c)=0.
\]
\end{ex}

This example illustrates an important subtlety: the Grothendieck construction is highly sensitive to the category under consideration.

We now enlarge the category under consideration. Instead of restricting ourselves to free modules, let us consider all finitely generated modules over a principal ideal domain.

Remarkably, the resulting Grothendieck group is still $\mathbb Z$.

\begin{prop}\label{prop: K0 for PID}
Let $R$ be a PID. Then
\[
K_0(R\textsf{\emph{-Mod}}^{fg})
\cong
\mathbb Z.
\]
\end{prop}

The proof again rests on the additivity of rank.

\begin{lemma}\label{lem: additivity of rank general}
Let $R$ be an integral domain, and let
\[
0\longrightarrow A\longrightarrow B\longrightarrow C\longrightarrow 0
\]
be a short exact sequence of finitely generated $R$-modules. Then
\[
\mathrm{rk}(B)=\mathrm{rk}(A)+\mathrm{rk}(C).
\]
\end{lemma}

\begin{proof}
We identify $A$ with a submodule of $B$ and $C$ with $B/A$.

There exists a linearly independent set $\{\alpha_1,\dots, \alpha_k\}$ in $A$, where $k=\mathrm{rk}(A)$, and a linearly independent set $\{[\gamma_1],\dots,[\gamma_\ell]\}$ in $C$, where $\ell = \mathrm{rk}(C)$ (here $[\gamma_i]$ denotes the coset of $\gamma_i$).

We claim that $\{\alpha_1,\dots,\alpha_k, \gamma_1,\dots, \gamma_\ell\}$ is linearly independent set in $B$. Indeed, suppose $a_1,\dots, a_k,c_1,\dots, c_\ell\in R$ and
\[\sum_{i=1}^k a_i\alpha_i + \sum_{j=1}^\ell c_i\gamma_i = 0.\]
Projecting onto $C = B/A$, it follows that
\[\sum_{j=1}^\ell c_i [\gamma_i] = 0\quad \text{in $C$},\]
so by linear independence all the $c_i$ equal zero. Hence
\[\sum_{i=1}^k a_i \alpha_i = 0 \quad \text{in $A\subset B$},\]
so all the $a_i$ equal zero. Hence $\{\alpha_1,\dots,\alpha_k, \gamma_1,\dots, \gamma_\ell\}$ is linearly independent in $B$, proving that $\mathrm{rk} B \ge k + \ell$. Now suppose that $\{\beta_1,\dots, \beta_{k+\ell+1}\}$ is a subset of $B$. We prove that it is linearly dependent. 

First of all, consider the projection on $C$ of this set. It must have a maximal linearly independent set, which without loss of generality we can take to be $\{[\beta_1],\dots,[\beta_m]\}$. Since $\ell = \mathrm{rk}(C)$, it follows that $m<\ell$. By maximality of this set, it follows that for each $i>m$, there exists $b_{i,1}\dots, b_{i,m}\in R$ such that
\[[\beta_i] - \sum_{j=1}^m b_{i,j}[\beta_j] = 0 \quad \text{in $C$}.\]
Hence
\[\beta_i' := \beta_i - \sum_{j=1}^m b_{i,j}\beta_j \in A,\]
so we get a set $\{\beta_{m+1}',\dots, \beta_{k+\ell+1}'\}$ of at least $\ell+1$ elements of $A$. This means that some nontrivial linear combination of them must equal zero. This extends to a nontrivial linear combination of the $\beta_i$ (since $\beta_i'$ is the only one with a nonzero $\beta_i$ coefficient). It follows that $\mathrm{rk} B< k+\ell +1$, so  $\mathrm{rk} B= k + \ell$. We are done.
\end{proof}

We are now ready to compute the Grothendieck group.

\begin{proof}[Proof of Proposition 1.8]
As before, rank defines a function
\[
\mathrm{rk}:R\textsf{-Mod}^{fg}\to \mathbb Z
\]
which is additive on short exact sequences by the preceding lemma. Hence, the universal property of $K_0$ yields a homomorphism
\[
\widetilde\delta:
K_0(R\textsf{-Mod}^{fg})
\to
\mathbb Z.
\]

Since $[R]$ maps to $1$, the map is surjective.
To prove injectivity, let
\[
x=[B]-[C]
\]
lie in the kernel of $\widetilde\delta$. Then
\[
\mathrm{rk}(B)=\mathrm{rk}(C).
\]

By the structure theorem for finitely generated modules over a PID,
\[
B
\cong
R^{\oplus \mathrm{rk}(B)}
\oplus
\frac{R}{(b_1)}
\oplus\cdots\oplus
\frac{R}{(b_m)}
\]
for suitable nonzero elements $b_1,\dots,b_m\in R$.

Now each ideal $(b_i)$ is isomorphic to $R$ as an $R$-module, so from the short exact sequence
\[
0\longrightarrow (b_i)\longrightarrow R\longrightarrow R/(b_i)\longrightarrow 0
\]
we obtain
\[
[R/(b_i)]=[R]-[(b_i)]=0.
\]

Using additivity, it follows that
\[
[B]=[R^{\oplus \mathrm{rk}(B)}].
\]
Similarly,
\[
[C]=[R^{\oplus \mathrm{rk}(C)}].
\]

Since $\mathrm{rk}(B)=\mathrm{rk}(C)$, we conclude $[B]=[C]$, and hence $x=0$.
Thus $\widetilde\delta$ is an isomorphism.
\end{proof}

\begin{rem}\label{rem: PID is neccasary}
    The reason we cannot say that $K_0(R\textsf{-Mod}^f)\cong \mathbb Z$ for arbitrary integral domains is that even though we have additivity of rank, we do not have the direct sum decomposition into the torsion part and the free part, like we have in a PID.
\end{rem}

This example illustrates a fundamentally different phenomenon from the free-module case. For finitely generated free modules, rank completely determines the module up to isomorphism. In contrast, finitely generated modules over a PID possess rich torsion structure invisible to rank alone.

The Grothendieck group therefore performs a genuine simplification: it collapses all torsion information while retaining only the additive invariant given by rank. In this sense, $K_0$ extracts the ``size'' of a module while discarding its finer internal structure.

Finally, since
\[
\mathsf{Ab}^{fg}\cong \mathbb Z\textsf{-Mod}^{fg},
\]
we immediately obtain the following corollary.

\begin{cor}\label{cor: ABfg}
We have \[
K_0(\mathsf{Ab}^{fg})\cong \mathbb Z.
\]
\end{cor}

\section{Jordan--H\"older Theory}

Up to this point, every Grothendieck group we have computed has turned out to be isomorphic to $\mathbb Z$. In each case, $K_0$ recovered a familiar numerical invariant such as dimension or rank. While these examples illustrate the mechanics of the construction, they do not yet reveal its full strength.

The situation becomes far more interesting when the underlying category contains genuinely nontrivial extension data. In such settings, Grothendieck groups encode how complicated objects are assembled from simpler ones. The key tool for understanding this phenomenon is the Jordan--H\"older theorem.

Let $R$ be a ring. Recall that a module is called \emph{simple} if it has no nontrivial proper submodules. Simple modules play the role of ``atomic'' objects in module theory: they cannot be decomposed any further.

\begin{defn}
A \emph{composition series} for an $R$-module $M$ is a chain of submodules
\[
M=M_0\supsetneq M_1\supsetneq \cdots \supsetneq M_m=0
\]
such that each quotient
\[
M_i/M_{i+1}
\]
is a simple $R$-module.
\end{defn}
The modules $M_i/M_{i+1}$ are called the \emph{composition factors} of the series, and the integer $m$ is called its \emph{length}.

Composition series should be viewed as analogues of prime factorizations: they decompose a complicated module into simple building blocks. A priori, however, it is not obvious that different composition series bear any relation to one another. The remarkable fact is that they do.

\begin{theorem}[Jordan--H\"older]
Any two composition series of a finite-length module have the same length and the same multiset of composition factors, up to isomorphism.
\end{theorem}

Thus, while a composition series itself is not unique, the collection of simple constituents appearing in it is uniquely determined. This suggests that the ``true'' additive information contained in a module is encoded by its simple composition factors.

Grothendieck groups formalize this intuition perfectly. Let
\[
R\textsf{-Mod}^f
\]
denote the category of finite-length $R$-modules.
The following theorem shows that every additive invariant of finite-length modules is completely determined by its values on simple modules.

\begin{theorem}\label{thm: uprop Jordan-Holder}
Let $G$ be an abelian group, and let $\delta$ be a function assigning an element of $G$ to every simple $R$-module. Then $\delta$ extends uniquely to a homomorphism
\[
\widetilde{\delta}:
K_0(R\textsf{\emph{-Mod}}^f)\to G.
\]
\end{theorem}

\begin{proof}
We define $\widetilde{\delta}$ explicitly.
Let
\[
M=M_0\supsetneq M_1\supsetneq \cdots \supsetneq M_m=0
\]
be a composition series for $M$. Define
\[
\widetilde{\delta}([M])
=
\delta(M_0/M_1)
+\cdots+
\delta(M_{m-1}/M_m).
\]

We must show that this definition is independent of the choice of composition series and that it respects the defining relations of $K_0$.

For the first point, the Jordan--H\"older theorem implies that any two composition series of $M$ have the same multiset of composition factors. Since addition in $G$ is commutative, the resulting value of $\widetilde{\delta}([M])$ is independent of the chosen series.
Next, suppose
\[
0\longrightarrow A\longrightarrow B\longrightarrow C\longrightarrow 0
\]
is a short exact sequence.
Choose composition series
\[
A=A_0\supsetneq A_1\supsetneq \cdots \supsetneq A_m=0
\]
\[
C=C_0\supsetneq C_1\supsetneq \cdots \supsetneq C_n=0.
\]

By the correspondence theorem, submodules of $C$ correspond to submodules of $B$ containing $A$. Thus the composition series for $C$ lifts to a chain
\[
B=C_0'\supsetneq C_1'\supsetneq \cdots \supsetneq C_n'=A.
\]

Combining this with the composition series for $A$, we obtain a composition series for $B$:
\[
B=C_0'
\supsetneq C_1'
\supsetneq \cdots
\supsetneq C_n'
=
A_0
\supsetneq A_1
\supsetneq \cdots
\supsetneq A_m
=
0.
\]

Moreover, $C_i'/C_{i+1}'\cong C_i/C_{i+1}$, so the composition factors of $B$ are precisely those of $A$ together with those of $C$.
Therefore,
\[
\widetilde{\delta}([B])
=
\widetilde{\delta}([A])
+
\widetilde{\delta}([C]),
\]
showing that $\widetilde{\delta}$ respects the defining relations of $K_0$. Hence it induces a homomorphism
\[
K_0(R\textsf{-Mod}^f)\to G.
\]
Uniqueness is immediate, since every finite-length module decomposes into composition factors.
\end{proof}

The reader may recognize this statement as a universal property. Indeed, it immediately implies the following structural description of the Grothendieck group.

\begin{cor}
The Grothendieck group
\[
K_0(R\textsf{\emph{-Mod}}^f)
\]
is the free abelian group generated by the isomorphism classes of simple $R$-modules.
\end{cor}

\begin{proof}
The preceding theorem is precisely the universal property characterizing a free abelian group on the set of simple modules.
\end{proof}

This result captures the philosophy of Grothendieck groups in a particularly transparent way. The category
\[
R\textsf{-Mod}^f
\]
may be extremely complicated: modules can possess intricate extension structures and many nonisomorphic modules may share the same composition factors. Yet the Grothendieck group forgets all extension data and remembers only the multiplicities of simple objects.

In other words, $K_0$ decategorifies the category of finite-length modules by collapsing every module to the formal sum of its simple constituents.

The theorem above is not specific to module categories. In fact, the same argument works in any finite-length abelian category.

\begin{defn}
An abelian category $\mathcal C$ is called \emph{finite-length} if every object admits a composition series of finite length.
\end{defn}

The Jordan--H\"older theorem extends to this setting, and the same proof yields the following general result.

\begin{theorem}
Let $\mathcal C$ be a finite-length abelian category. Then
\[
K_0(\mathcal C)
\]
is the free abelian group generated by the isomorphism classes of simple objects of $\mathcal C$.
\end{theorem}

Thus, in any finite-length category, the Grothendieck group completely reduces the category to its simple building blocks.

We conclude with an example illustrating how nontrivial Grothendieck groups can arise even in classical algebra.

\begin{ex}\label{ex: Abf}
Let
\[
\mathsf{Ab}^f
\]
denote the category of finite abelian groups.
Since every strictly descending chain of subgroups decreases the order of the group, every finite abelian group has finite length. Thus $\mathsf{Ab}^f$ is a finite-length category.
The simple objects of $\mathsf{Ab}^f$ are precisely the cyclic groups $\mathbb Z/p\mathbb Z$ for primes $p$. Consequently,
\[
K_0(\mathsf{Ab}^f)
\]
is the free abelian group generated by the classes $[\mathbb Z/p\mathbb Z]$.

This already contrasts sharply with our earlier examples, where $K_0\cong \mathbb Z$.

The Grothendieck group also naturally encodes the order of a finite abelian group. Define
\[
\delta(\mathbb Z/p\mathbb Z)=p
\]
as an element of the multiplicative group $\mathbb Q^\times$. By the universal property above, this extends uniquely to a homomorphism
\[
\widetilde{\delta}:
K_0(\mathsf{Ab}^f)\to \mathbb Q^\times.
\]

If $A=A_0\supsetneq A_1\supsetneq \cdots \supsetneq A_n=0$ is a composition series for $A$, then
\[
\widetilde{\delta}([A])
=
\prod_{i=0}^{n-1}|A_i/A_{i+1}|.
\]

Since the order of a finite group is multiplicative in short exact sequences,
\[
|B|=|A|\cdot |B/A|,
\]
it follows inductively that
\[
\widetilde{\delta}([A])=|A|.
\]

Thus the ordinary notion of cardinality factors naturally through the Grothendieck group.
\end{ex}

\medskip
This example illustrates the broader philosophy beautifully: Grothendieck groups transform complicated algebraic structures into linear combinations of simple building blocks, while preserving fundamental additive information.

\section{The Euler Characteristic}

One of the remarkable features of Grothendieck groups is that they naturally produce alternating-sum invariants. The most classical and important example of such an invariant is the \emph{Euler characteristic}.

Historically, the Euler characteristic arose in topology long before the development of homological algebra or $K$-theory. In its earliest form, it was defined for polyhedra by the formula
\[
\chi = V-E+F,
\]
where $V$, $E$, and $F$ denote the numbers of vertices, edges, and faces, respectively.

Euler's astonishing discovery was that this quantity does not depend on the particular shape of the polyhedron, but only on its underlying topology. In particular, every convex polyhedron satisfies
\[
V-E+F=2.
\]

For example, a cube has
\[
V=8,\qquad E=12,\qquad F=6,
\]
so $8-12+6=2$.
The same computation works for every convex polyhedron, no matter how complicated its geometry may be.

More generally, the Euler characteristic detects global topological structure. A torus, for instance, has Euler characteristic $0$, reflecting the presence of a ``hole'' in the surface. In this sense, the Euler characteristic is one of the earliest examples of a topological invariant.

The natural setting for this construction is the theory of \emph{CW-complexes}. Roughly speaking, a finite CW-complex is built by gluing together cells (topological spheres) of various dimensions:
\[
0\text{-cells},\ 1\text{-cells},\ 2\text{-cells},\ \dots
\]
A polyhedron, for example, is assembled from vertices (0-cells), edges (1-cells), and polygonal faces (2-cells). 
If $X$ is a finite CW-complex and $k_i$ denotes the number of $i$-dimensional cells, then the Euler characteristic of $X$ is defined by
\[
\chi(X)=k_0-k_1+k_2-k_3+\cdots.
\]

At first sight this formula appears purely combinatorial. However, one of the great insights of algebraic topology is that it can be reinterpreted homologically. From this perspective, the Euler characteristic becomes an instance of a far more general categorical construction.

Indeed, Grothendieck groups allow us to define Euler characteristics in \emph{any} abelian category.

\begin{defn}
Let
\[
V_\bullet:
0\longrightarrow
V_N
\xrightarrow{\alpha_N}
V_{N-1}
\xrightarrow{\alpha_{N-1}}
\cdots
\xrightarrow{\alpha_1}
V_0
\longrightarrow 0
\]
be a bounded complex in an abelian category $\mathcal C$. The \emph{Euler characteristic} of $V_\bullet$ is the element
\[
\chi(V_\bullet)
=
\sum_i (-1)^i [V_i]
\in
K_0(\mathcal C).
\]
\end{defn}

This definition may initially seem unrelated to topology. The connection emerges once we learn how to encode a CW-complex algebraically.

Let $X$ be a finite CW-complex. Denote by $\sigma_n(X)$, the set of $n$-dimensional cells of $X$. We then define $C_n(X)$ to be the free abelian group generated by these cells.
For example, if $X$ is a polyhedron, then:
\begin{itemize}
    \item $C_0(X)$ is generated by the vertices,
    \item $C_1(X)$ is generated by the edges,
    \item $C_2(X)$ is generated by the faces.
\end{itemize}

Thus, $\mathrm{rk}(C_n(X)) = |\sigma_n(X)| = k_n$.
Since we are working in the category of finitely generated abelian groups, Corollary 1.9 implies that
\[
K_0(\mathsf{Ab}^{fg})\cong \mathbb Z,
\]
with the class of a group identified with its rank. Consequently,
\[
[C_n(X)] = k_n,
\]
so the classical Euler characteristic becomes
\[
\chi(X)=\sum_n (-1)^n [C_n(X)].
\]

To turn the groups $C_n(X)$ into a complex, we introduce the \emph{boundary maps}
\[
\partial_n:C_n(X)\to C_{n-1}(X).
\]

Intuitively, $\partial_n$ sends a cell to its boundary. For example:
\begin{itemize}
    \item the boundary of an edge is the difference of its two endpoints,
    \item the boundary of a polygonal face is the sum of its edges,
    \item the boundary of a solid three-dimensional cell is its surrounding surface.
\end{itemize}

These maps assemble into a chain complex
\[
C_\bullet(X):
0\longrightarrow
C_N(X)
\xrightarrow{\partial_N}
C_{N-1}(X)
\xrightarrow{\partial_{N-1}}
\cdots
\xrightarrow{\partial_1}
C_0(X)
\longrightarrow 0.
\]

The crucial fact is that
\[
\partial_n\circ \partial_{n+1}=0.
\]

Geometrically, this expresses the intuitive principle that ``the boundary of a boundary is empty.'' For instance, the boundary of a solid ball is a sphere, and the sphere itself has no boundary.

Thus $C_\bullet(X)$ is genuinely a complex, and its Euler characteristic is precisely the classical Euler characteristic of the space:
\[
\chi(X)=\chi(C_\bullet(X)).
\]

The power of the categorical viewpoint comes from the fact that Euler characteristics depend only on the \emph{homology} of the complex.

\begin{prop}\label{prop: Eulerchar}
Let
\[
V_\bullet:
0\to V_N\to \cdots \to V_0\to 0
\]
be a bounded complex in an abelian category. Then
\[
\chi(V_\bullet)
=
\sum_{i=0}^N (-1)^i [H_i(V_\bullet)].
\]

In particular, if $V_\bullet$ is exact, then $\chi(V_\bullet)=0$.
\end{prop}

\begin{proof}
    There is nothing to show for $N = 0$. In the case $N=1$, we have
    \[\chi(V_\bullet) = [V_0]-[V_1],\]
    and
    \begin{align*}
        [H_0(V_{\bullet})] - [H_1(V_\bullet)] = [V_0/\mathrm{im}(\alpha_1)] - [\ker(\alpha_1)].
    \end{align*}
    There are short exact sequences
    \[0 \longrightarrow \mathrm{im}(\alpha_1) \longrightarrow V_0 \longrightarrow V_0/\mathrm{im}(\alpha_1) \longrightarrow 0\]
    and
    \[0 \longrightarrow \mathrm{ker}(\alpha_1) \longrightarrow V \longrightarrow \mathrm{im}(\alpha_1) \longrightarrow 0,\]
    giving the relations
    \[[V_0/\mathrm{im}(\alpha_1)]  = [V_0]-[\mathrm{im}(\alpha_1)]\qquad \text{and}\qquad [\mathrm{im}(\alpha_1)] + [\ker(\alpha_1)] = [V_1].\] 
    Combined, these imply that 
    \[[V_0/\mathrm{im}(\alpha_1)] - [\ker(\alpha_1)] = [V_0]-[V_1],\]
    proving the claim for $N=1$.

    For general $N$, we use induction. Given $V_\bullet$ of length $N$, we may assume that the result is true for shorter complexes. 

    We naturally obtain the truncation complex
    \[V_\bullet ' : \qquad 0 \longrightarrow V_{N-1}\xrightarrow{\ \alpha_{N-1}\ } \cdots \xrightarrow{\ \, \alpha_1 \, \ }V_0 \longrightarrow 0.\]
    Then
    \begin{align*}
        \chi(V_\bullet) = \chi(V_\bullet') + (-1)^N [V_N].\tag{$\dag$}
    \end{align*}
    To compute the homology of $V_\bullet$ in terms of the homology of $V_\bullet'$, observe that
    \[H_i(V_\bullet) = H_i(V_\bullet') \quad \text{for $0\le i\le N-2$}\]
    and
    \[H_{N-1}(V_\bullet') = \ker(\alpha_{N-1}),\quad H_{N-1}(V_\bullet) = \frac{\ker(\alpha_{N-1})}{\mathrm{im}(\alpha_N)},\quad H_N(V_\bullet) = \ker(\alpha_N).\]
    It follows for the same reason as the $N=1$ case that
    \[[H_{N-1}(V_\bullet)] -[H_N(V_\bullet)] = [H_{N-1}(V_\bullet')]-[V_N].\]
    By induction, we have
    \begin{align*}
        \chi(V_\bullet') &= \sum_{i} (-1)^i [H_i(V_\bullet')]\\ &= (-1)^{N-1}[H_{N-1}(V_\bullet)] + (-1)^{N}[H_{N}(V_\bullet)] -(-1)^N[V_N]+ \sum_{i\le N-2}(-1)^i[H_i(V_\bullet)].
    \end{align*}
    Rearranging and applying $(\dag)$, we conclude.
\end{proof}

\begin{rem}\label{rem: uprop of Eulerchar} In the context of \cref{prop: K0 of R-Free}, if we define
\[\chi_G(V_\bullet) = \sum_i (-1)^i \delta(V_i),\]
then by linearity the unique homomorphism $K_0(\mathcal{C}) \to G$ maps $\chi(V_\bullet)$ to $\chi_G(V_\bullet)$. 

This shows how the Euler characteristic is a ``universal Euler charateristic.''
\end{rem}

\medskip
Going back to the topological discussion, this proposition has profound consequences.

For a CW-complex $X$, the homology groups
\[
H_n(X):=H_n(C_\bullet(X))
\]
measure the global topology of the space. Their ranks
\[
b_n=\mathrm{rk}(H_n(X))
\]
are called the \emph{Betti numbers} of $X$.
The proposition therefore implies the celebrated formula
\[
\chi(X)
=
b_0-b_1+b_2-b_3+\cdots.
\]

This expresses the Euler characteristic entirely in terms of homology.
Intuitively:
\begin{itemize}
    \item $b_0$ counts connected components,
    \item $b_1$ counts one-dimensional holes or loops,
    \item $b_2$ counts two-dimensional cavities,
    \item and higher Betti numbers measure higher-dimensional holes.
\end{itemize}

For example:
\begin{itemize}
    \item a sphere satisfies $b_0=1$, $b_1=0$, $b_2=1$, so
    \[
    \chi(S^2)=1-0+1=2;
    \]
    \item a torus satisfies $b_0=1$, $b_1=2$, $b_2=1$, so
    \[
    \chi(T^2)=1-2+1=0.
    \]
\end{itemize}

Thus the Euler characteristic, originally defined as a combinatorial alternating sum, is ultimately governed by homological algebra.

From the perspective of Grothendieck groups, this phenomenon becomes almost inevitable: exact pieces cancel in $K_0$, leaving behind only homological information. In this way, the Euler characteristic provides one of the clearest illustrations of the philosophy underlying Grothendieck's construction.

For a more detailed treatment of CW-complexes and homology theory, see \cite{hatcher2002algebraic}.

\section{Representation Rings}

One of the recurring themes of this paper has been that Grothendieck groups become especially meaningful when the underlying category carries additional structure beyond direct sum. Representation theory provides one of the most elegant examples of this phenomenon.

Recall that if $G$ is a finite group, then a representation of $G$ over $\mathbb C$ is a homomorphism
\[
\rho:G\to \mathrm{GL}(V),
\]
where $V$ is a finite-dimensional complex vector space. The category of such representations, denoted
\[
\mathrm{Rep}_{\mathbb C}(G),
\]
is extraordinarily rich: representations encode symmetries of algebraic, geometric, and combinatorial objects, while simultaneously translating group-theoretic questions into problems in linear algebra.

Unlike many of the categories considered previously, $\mathrm{Rep}_{\mathbb C}(G)$ possesses \emph{two} natural operations. In addition to direct sum, representations may also be tensor multiplied. If $U$ and $V$ are representations of $G$, then the actions on
\[
U\oplus V
\qquad\text{and}\qquad
U\otimes V
\]
are defined by
\[
g(u\oplus v)=gu\oplus gv,
\]
and
\[
g(u\otimes v)=gu\otimes gv,
\]
respectively.

The direct sum operation makes the set of isomorphism classes of representations into a commutative monoid. As before, we may formally adjoin inverses.

\begin{defn}
The \emph{representation ring} $R(G)$ is the Grothendieck group of the monoid of finite-dimensional complex representations of $G$ under direct sum. Equivalently,
\[
R(G)
\]
is the abelian group generated by symbols $[V]$, one for each representation $V$, subject to the relations
\[
[A\oplus B]=[A]+[B].
\]
\end{defn}

The tensor product enriches this construction considerably. Since tensor product distributes over direct sum,
\[
(U\oplus V)\otimes W
\cong
(U\otimes W)\oplus (V\otimes W),
\]
the operation
\[
[V]\cdot [W]=[V\otimes W]
\]
extends to a multiplication operation on $R(G)$. The trivial one-dimensional representation serves as the multiplicative identity. Consequently, $R(G)$ is not merely an abelian group, but a commutative ring.

At first glance, however, the additive structure of $R(G)$ may appear somewhat disappointing. Maschke's theorem implies that every finite-dimensional complex representation of a finite group decomposes as a direct sum of irreducible representations. Thus, if
\[
V_1,\dots,V_r
\]
are the irreducible representations of $G$, every representation has the form
\[
V_1^{\oplus n_1}\oplus \cdots \oplus V_r^{\oplus n_r}.
\]

It follows that the monoid of representations is isomorphic to
\[
\mathbb N^r,
\]
and hence
\[
R(G)\cong \mathbb Z^r
\]
as abelian groups.

From this perspective, the additive structure of $R(G)$ simply records the multiplicities of irreducible representations. The truly interesting information lies in the multiplicative structure: how tensor products of irreducible representations decompose into irreducibles.

To understand this structure, we introduce one of the central ideas of representation theory: characters.
For a representation
\[
\rho:G\to GL_n(\mathbb C),
\]
the associated \emph{character} is the function
\[
\chi_\rho:G\to \mathbb C
\]
defined by
\[
\chi_\rho(g)=\operatorname{tr}(\rho(g)).
\]

At first sight, the character appears to be a remarkably crude invariant: it records only the traces of the representing matrices. Nevertheless, characters encode an astonishing amount of information about a representation.

\begin{prop}\label{prop: characters}
Let $\chi$ be the character of an $n$-dimensional representation of $G$. Then:
\begin{enumerate}[label=$(\arabic*)$]
    \item $\chi(1)=n$,
    \item $\chi(g^{-1})=\overline{\chi(g)}$ for all $g\in G$,
    \item $\chi(st)=\chi(ts)$ for all $s,t\in G$.
\end{enumerate}
\end{prop}

The third property reflects the cyclic invariance of trace:
\[
\operatorname{tr}(AB)=\operatorname{tr}(BA).
\]
Functions satisfying
\[
f(st)=f(ts)
\]
for all $s,t\in G$ are called \emph{class functions}, since they are constant on conjugacy classes.

\medskip
{\bf Note:} For proofs and more detailed exposition of the above and the following results, see Chapter 2 of \cite{serre1977representations}.

\medskip
Characters interact beautifully with the additive and multiplicative operations on representations. If $V_1$ and $V_2$ have characters $\chi_1$ and $\chi_2$, then
\[
\chi_{V_1\oplus V_2}
=
\chi_1+\chi_2,
\]
while
\[
\chi_{V_1\otimes V_2}
=
\chi_1\chi_2.
\]

Thus direct sum and tensor product translate into ordinary addition and multiplication of functions. In this way, characters ``linearize'' representation theory.

The remarkable fact is that characters lose essentially no information.

\begin{theorem}
Two finite-dimensional complex representations of $G$ are isomorphic if and only if they have the same character.
\end{theorem}

Consequently, characters completely classify representations up to isomorphism.

The irreducible representations play a particularly distinguished role. If $G$ has $r$ conjugacy classes, then $G$ possesses exactly $r$ irreducible representations over $\mathbb C$. Let
\[
\chi_1,\dots,\chi_r
\]
denote the corresponding irreducible characters.

The following theorem is one of the cornerstones of character theory.

\begin{theorem}\label{thm: class functions}
Every class function
\[
f:G\to \mathbb C
\]
can be expressed uniquely as a complex linear combination of the irreducible characters.
\end{theorem}

In other words, the irreducible characters form a basis for the vector space of class functions on $G$.
This immediately clarifies the structure of the representation ring.

\begin{cor}
The representation ring $R(G)$ is naturally isomorphic to the ring of class functions.
\end{cor}

\begin{proof}
Characters are additive under direct sum and multiplicative under tensor product, so taking characters defines a semiring homomorphism from representations of $G$ to class functions on $G$.

By the previous theorem, every character decomposes uniquely as a sum of irreducible characters, and by \cref{prop: characters}, every linear combination of irreducible characters is a class function. 

Now every element of $R(G)$ is a linear combination of the $[V_i]$, so the map above extends to a ring homomorphism from $R(G)$ to the $\mathbb Z$-span of the irreducible characters $\chi_1,\dots, \chi_n$.

It is evident that this map is an isomorphism (of rings), since $R(G)$ is the free abelian group generated by the simple objects $[V_i]$. By \cref{thm: class functions}, the $\mathbb Z$-span of the irreducible characters is just the space of class functions, which completes the proof.
\end{proof}

Representation rings provide one of the clearest illustrations of the philosophy of decategorification. The complicated tensor category $\mathrm{Rep}_{\mathbb C}(G)$ is replaced by a commutative ring encoding how irreducible representations combine under tensor product. Although much categorical structure is lost in passing to $R(G)$, the resulting ring still retains surprisingly deep information about the representation theory of the group.

More broadly, representation rings foreshadow many later developments in $K$-theory. The passage from a tensor category to its Grothendieck ring reappears throughout modern mathematics, from algebraic geometry and topology to quantum groups and categorification theory. In this sense, $R(G)$ is not merely an isolated construction, but one of the earliest and most concrete manifestations of a pervasive idea.

\section{Grothendieck Group of a Ring}
Let $R$ be a \textit{possibly noncommutative} ring, and consider the category $\mathrm{Proj}(R)$ of finitely generated \textit{projective} $R$-modules.

The Grothendieck group $K_0(R)$ is generated by isomorphism classes $[P]$, for $P\in \mathrm{Proj}(R)$, subject to the relations
\[[P\oplus Q]=[P]+[Q].\]
Since the set of $[P]$ forms a monoid under direct sums, this is actually just an example of the Grothendieck group of a monoid introduced at the beginning of the chapter.

Note that this definition is the same as the Grothendieck Group $K_0(\mathrm{Proj}(R))$ introduced in \S 1: even though the addition operation here only needs to factor through \textit{split} exact sequences, any exact sequence of projective modules automatically splits. 

If $R$ is commutative, we can upgrade $K_0(R)$ to have a ring structure via the tensor product. Indeed, since $\otimes$ is commutative, distributes over $\oplus$, and $1 = [R]$ acts as the multiplicative identity, the operation $\times$ defined by \[[A]\times [B]=[A\otimes B]\] gives a valid multiplicative structure on $K_0(R)$.   

We can slightly improve Lemma 1.2.
\begin{lemma}
    Every element of $K_0(R)$ can be written as $[A]-[R^n]$, for $A,R^n\in \mathrm{Proj}(R)$.
\end{lemma}
\begin{proof}
    By Lemma 1.2, every element can be written as $[B]-[C]$. By Proposition VIII.6.4 (or more precisely its generalization to noncommutative rings), it follows that we may write $R^n = C\oplus K$ for some $n$, and another projective $R$-module $K$. Hence,
    \[[B]-[C] = [B]-([R^n]-[K]) = [B\oplus K]-[R^n],\]
    which is the desired form.
\end{proof}

If $R$ is a field, or more generally a PID, then every projective $R$-module is free. Hence $\mathrm{Proj}(R) \cong R\mathsf{-Free}^{fg}$ and by Proposition 1.5 we have $K_0(R) \cong \mathbb Z$.

This $\mathbb Z$-structure usually appears in the Grothendieck group $K_0(R)$, coming from the free modules.

\begin{lemma}
    The monoid map $\mathbb N\to \mathrm{Proj}(R)$ sending $n$ to $R^n$ induces a group homomorphism $\mathbb Z\to K_0(R)$. This map is injective if and only if $R$ satisfies the IBN property.
\end{lemma}
\begin{proof}
    The first statement follows since $\mathbb Z$ is the group completion of $\mathbb N$ and $K_0(R)$ is the group completion of $\mathrm{Proj}(R)$. 

    For the second point, observe that the group homomorphism is injective if and only if the monoid map $\mathbb N\to \mathrm{Proj}(R)$ is. Injectivity of this map means that if $m\ne n$, then $R^n\not\cong R^m$. This is same as the IBN property.
\end{proof}

We would like to get rid of this structure, because ``rank'' is not really telling us anything interesting. We can do this by simply taking the cokernel of the map $\mathbb Z\to K_0(R)$. And in the case where $R$ is ``sufficiently nice'' (e.g. an integral domain), this indeed is a good construction.

However, when $R$ is not an integral domain, the rank of a projective $R$-module is no longer an integer; it is a continuous function from $\mathrm{Spec}(R) \to \mathbb Z$. 

For the remainder of the section, suppose $R$ is a commutative ring.

We let $H_0(R) = [\mathrm{Spec}(R),\mathbb Z]$ to be the space of all such continuous maps. If $R$ is an integral domain, then $\mathrm{Spec}(R)$ is connected, so $H_0(R)\cong \mathbb Z$. Then we have a rank map $\mathrm{Proj}(R) \to H_0(R)$. By the properties
\[\mathrm{rk}(P\oplus Q) = \mathrm{rk}(P) + \mathrm{rk}(Q) \quad \text{and}\quad \mathrm{rk}(P\otimes Q) = \mathrm{rk}(P)\cdot \mathrm{rk}(Q),\]
it follows that rank defines a semiring map. 

It turns out that we can identify $H_0(R)$ with a subgroup of $K_0(R)$: for every $f\in [\mathrm{Spec}(R), \mathbb N]$ we can associate a projective module $R^f$ with rank $f$. This means that $H_0(R)$ is in fact a direct summand of $K_0(R)$, so we can write
\[K_0(R) \cong H_0(R)\oplus \widetilde{K}_0(R),\]
for an ideal $\widetilde{K}_0(R)$ of $K_0(R)$.
\begin{defn}
    $\widetilde{K}_0(R)$ is called the \textit{reduced Grothendieck group} of $R$.
\end{defn}
This group shows up in many places.

\begin{ex}
    For an odd prime $p$, let $R$ be the quotient ring \[\frac{\mathbb Z[x]}{(1+x+\cdots + x^{p-1})}.\] The reduced Grothendieck group, $\widetilde{K}_0(R) =K(R)/\mathbb Z$, is a finite group, and it is related to Fermat's Last Theorem. Krummer proved that Fermat's equation \[x^p + y^p = z^p\] has no nontrivial integer solutions, if the group $\widetilde{K}_0(R)$ has no $p$-torsion.
\end{ex}

\medskip
But most importantly, these Grothendieck groups are the beginning of the study of topological $K$-theory.

\section{An Introduction to $K$-Theory}

One of the central themes of modern mathematics is the attempt to classify geometric or algebraic objects up to suitable notions of equivalence. For example, one may wish to classify vector bundles on a topological space, modules over a ring, or coherent sheaves on an algebraic variety. Such classification problems are usually far too complicated to solve directly. The philosophy of $K$-theory is to replace these objects by algebraic invariants that retain essential structural information while being easier to study.

The origins of $K$-theory lie in the work of Grothendieck, who introduced Grothendieck groups in his proof of the Grothendieck--Riemann--Roch theorem. Since then, $K$-theory has developed into a vast subject with deep connections to algebra, topology, geometry, number theory, and mathematical physics.

In this section we discuss \emph{topological $K$-theory}, which builds on the discussion in \S 4. Our presentation loosely follows \cite{karoubi2006k}. Before introducing $K$-theory itself, we first need some basic ideas from algebraic topology, especially homotopy theory.

\subsection{A Primer on Homotopy Theory}

Homotopy theory studies topological spaces up to continuous deformation. Roughly speaking, two spaces are considered equivalent if one can be continuously deformed into the other without tearing or gluing.
The basic notion is that of a \emph{homotopy} between maps. Let $X$ and $Y$ be topological spaces, and let
\[
f,g:X\to Y
\]
be continuous maps. We say that $f$ and $g$ are \emph{homotopic} if there exists a continuous map
\[
F:X\times [0,1]\to Y
\]
such that
\[
F(x,0)=f(x), \qquad F(x,1)=g(x)
\]
for all $x\in X$.

Intuitively, the parameter $t\in [0,1]$ plays the role of ``time,'' and the map $F$ describes a continuous deformation of $f$ into $g$.

This notion leads naturally to a corresponding notion for spaces themselves. We say that topological spaces $X$ and $Y$ are \emph{homotopy equivalent} if there exist continuous maps
\[
f:X\to Y,
\qquad
g:Y\to X
\]
such that
\[
g\circ f \simeq \id_X,
\qquad
f\circ g \simeq \id_Y,
\]
where $\simeq$ denotes homotopy. Thus $X$ and $Y$ have the same ``shape'' from the viewpoint of homotopy theory.

For example, a circle with a line segment attached to it is homotopy equivalent to a circle: the extra segment can simply be contracted to a point. On the other hand, a circle is not homotopy equivalent to a sphere, since the circle contains a one-dimensional ``hole'' that the sphere does not.

It is not difficult to show that both notions of homotopy define equivalence relations.

\medskip

In algebraic topology, one often works with \emph{based spaces}. A based space is a topological space together with a distinguished point $x_0\in X$, called the \emph{base point}. A map
\[
f:(X,x_0)\to (Y,y_0)
\]
is called \emph{base-point preserving} if
\[
f(x_0)=y_0.
\]

When defining homotopies between such maps, we additionally require that the base point remain fixed throughout the deformation:
\[
F(x_0,t)=y_0
\qquad
\text{for all } t\in [0,1].
\]

The most important examples arise from maps out of spheres. Recall that the $p$-dimensional sphere is defined by
\[
S^p
=
\left\{
(x_1,\dots,x_{p+1})\in \mathbb R^{p+1}
:
x_1^2+\cdots+x_{p+1}^2=1
\right\}.
\]

For small values of $p$, these are familiar objects:

\begin{itemize}
    \item $S^0=\{-1,1\}$ consists of two isolated points;
    \item $S^1$ is the circle;
    \item $S^2$ is the ordinary sphere in $\mathbb R^3$;
    \item $S^3$ is a three-dimensional sphere sitting inside $\mathbb R^4$.
\end{itemize}

\medskip

We now extract algebraic invariants from these ideas. Let $X$ be a based topological space with base point $x_0$. The $p$-th \emph{homotopy group} of $X$, denoted
\[
\pi_p(X,x_0),
\]
is the set of homotopy classes of base-point preserving maps
\[
S^p\to X.
\]

When $p=0$, a map $S^0\to X$ simply picks out a point of $X$. Passing to homotopy classes identifies points lying in the same path-connected component, so
\[
\pi_0(X)
\]
is the set of path-connected components of $X$.

The first genuinely interesting case is $p=1$. Then we are considering loops in $X$ based at $x_0$:
\[
\gamma:[0,1]\to X,
\qquad
\gamma(0)=\gamma(1)=x_0.
\]

Given two such loops, we may traverse one after the other, producing a new loop. This operation defines a group structure on
\[
\pi_1(X,x_0),
\]
called the \emph{fundamental group} of $X$.
For example:

\begin{itemize}
    \item $\pi_1(S^2)=0$, since every loop on a sphere can be contracted to a point;
    \item $\pi_1(S^1)\cong \mathbb Z$, since loops around the circle are classified by their winding number.
\end{itemize}

More generally, one can define composition laws on all higher homotopy groups
\[
\pi_2(X),\pi_3(X),\dots.
\]
A remarkable fact is that for $p\ge2$, the groups $\pi_p(X)$ are always abelian.

\medskip

Like the homology groups introduced in \S 4, the homotopy groups measure the ``holes'' of a space. However, they are much subtler invariants. For instance,
\[
H_1(X)\cong \pi_1(X)^{ab},
\]
the abelianization of the fundamental group. Thus homology forgets the noncommutative structure of loops and therefore loses information.

Homotopy groups are among the most important invariants in topology. One of their most famous appearances is in the \emph{Poincar\'e conjecture}. The conjecture asserts that a closed simply connected $3$-manifold is homeomorphic to the $3$-sphere $S^3$. In other words, certain homotopy-theoretic data completely determine the topology of the manifold.

\medskip

Our main interest, however, is a remarkable connection between homotopy theory and $K$-theory known as \emph{Bott periodicity}. This theorem, discovered by Bott, reveals an unexpected periodic pattern in the homotopy groups of classical Lie groups.

Consider the general linear groups
\[
GL_n(\mathbb C),
\]
consisting of invertible complex $n\times n$ matrices. Since $GL_n(\mathbb C)$ is an open subset of the vector space $M_n(\mathbb C)$, it naturally carries the structure of a smooth manifold. Consequently, we may study its homotopy groups
\[
\pi_p(GL_n(\mathbb C)).
\]

There are natural inclusions
\[
GL_n(\mathbb C)\hookrightarrow GL_{n+1}(\mathbb C),
\]
obtained by sending a matrix $A$ to the block matrix
\[
\begin{pmatrix}
A & 0\\
0 & 1
\end{pmatrix}.
\]

A fundamental fact is that these inclusions eventually stabilize the homotopy groups:
\[
\pi_p(GL_n(\mathbb C))
\cong
\pi_p(GL_{n+1}(\mathbb C))
\]
for sufficiently large $n$. We therefore write
\[
\pi_p(\mathrm{GL}(\mathbb C))
\]
for this stable value.

Bott's remarkable discovery was that these stable groups are periodic.

\begin{theorem}[Bott Periodicity]\label{thm: bott}
The stabilized homotopy groups satisfy
\[
\pi_p(\mathrm{GL}(\mathbb C))
\cong
\pi_{p+2}(\mathrm{GL}(\mathbb C)).
\]
\end{theorem}

This theorem is one of the cornerstones of modern topology, and it lies at the heart of topological $K$-theory.

\subsection{Bott Periodicity and $K$-Theory}

We now explain how Bott periodicity is related to Grothendieck groups.

Recall from the previous section that to a ring $R$ one associates its Grothendieck group $K_0(R)$, constructed from finitely generated projective $R$-modules. The philosophy of topological $K$-theory is that vector bundles over a space $X$ should be studied through the Grothendieck group of an associated ring of functions on $X$. Let
\[
C_{\mathbb C}(S^p)
\]
denote the ring of complex-valued continuous functions on the sphere $S^p$. Since this ring is an integral domain, we have a decomposition
\[
K_0(C_{\mathbb C}(S^p))
\cong
\mathbb Z\oplus \widetilde K_0(C_{\mathbb C}(S^p)),
\]
where $\widetilde K_0$ denotes the reduced Grothendieck group.

The striking connection with homotopy theory is the following theorem.

\begin{theorem}\label{thm: homotopy and K-theory}
For every $p\ge1$,
\[
\pi_{p-1}(\mathrm{GL}(\mathbb C))
\cong
\widetilde K_0(C_{\mathbb C}(S^p)).
\]
\end{theorem}

Thus the stable homotopy groups of the infinite general linear group can be interpreted as Grothendieck groups of rings of functions on spheres.

From this perspective, Bott periodicity becomes a statement about $K$-theory.

\begin{cor}
For every $p$,
\[
\widetilde K_0(C_{\mathbb C}(S^p))
\cong
\widetilde K_0(C_{\mathbb C}(S^{p+2})).
\]
\end{cor}

In other words, the $K$-theory of spheres is periodic with period $2$.

There are analogous statements over $\mathbb R$, where the corresponding periodicity has period $8$ rather than $2$. This phenomenon, known as \emph{real Bott periodicity}, plays a fundamental role throughout topology and operator algebras.

\subsection{Higher $K$-Theory}
Throughout this paper, we have denoted Grothendieck groups by $K_0$. This notation is not accidental: it suggests the existence of further groups
\[
K_1,\ K_2,\ K_3,\ \dots,
\]
and indeed these groups form the subject of \emph{higher $K$-theory}.

The passage from $K_0$ to higher $K$-groups is one of the great conceptual expansions of Grothendieck's original idea. What began as a method for linearizing categories eventually developed into a deep theory connecting topology, algebra, number theory, operator algebras, and arithmetic geometry.

From a modern perspective, higher $K$-theory measures phenomena invisible to $K_0$. While $K_0$ records additive information about objects, the groups $K_1,K_2,\dots$ detect increasingly subtle geometric and homotopical structure.

There are many approaches to higher $K$-theory. Quillen's definition uses homotopy theory and classifying spaces; Waldhausen's construction uses categories of complexes; algebraic geometers study higher $K$-groups of schemes; and topologists interpret them through Bott periodicity and generalized cohomology theories. 

Rather than attempting a complete development, we conclude with a brief glimpse of the topological viewpoint, continuing the discussion of Bott periodicity from the previous section.


The first step is to extend the definition of Grothendieck groups to rings without identity.

Let $A$ be a $k$-algebra, where $k$ is a commutative ring with identity. Even if $A$ itself has no multiplicative identity, we can formally adjoin one by defining
\[
A^+=A\times k
\]
with multiplication
\[
(a,\lambda)(a',\lambda')
=
(aa'+\lambda a'+\lambda' a,\lambda\lambda').
\]

The element $(0,1)$ serves as the identity of $A^+$.
We then define
\[
K_0(A)
\]
to be the kernel of the natural map
\[
K_0(A^+)\to K_0(k).
\]

Intuitively, this construction isolates the information contributed by $A$ itself, removing the artificially adjoined identity component.

A remarkable fact is that this definition is canonical.

\begin{theorem}
The definition of $K_0(A)$ above is independent of the choice of the ground ring $k$.
\end{theorem}

One important consequence is functoriality: every ring homomorphism
\[
A\to B
\]
induces a homomorphism
\[
K_0(A)\to K_0(B),
\]
even if the rings are nonunital or the map does not preserve identities.

Moreover, this generalized Grothendieck group satisfies a weak form of exactness.

\begin{theorem}\label{thm: exactness of K0}
An exact sequence
\[
0\longrightarrow A\longrightarrow B\longrightarrow C\longrightarrow 0
\]
of rings induces an exact sequence
\[
K_0(A)\longrightarrow K_0(B)\longrightarrow K_0(C).
\]
\end{theorem}

At this point, a natural question arises. In homological algebra, exact sequences rarely appear in isolation; they are usually part of a longer exact sequence extending infinitely in both directions. Is there a way to continue the sequence above further to the left?

Higher $K$-theory provides precisely such an extension.


To define higher $K$-groups topologically, we need a setting in which ideas from topology and analysis can interact algebraically. This leads naturally to Banach algebras.

\begin{defn}
A \emph{Banach algebra} over $k=\mathbb R$ or $\mathbb C$ is a $k$-algebra $A$ equipped with a norm $\|\cdot\|$ such that:
\begin{enumerate}[label=$(\arabic*)$]
    \item $A$ is complete with respect to the metric
    \[
    d(a,b)=\|a-b\|;
    \]
    \item multiplication is continuous in the sense that
    \[
    \|ab\|\le \|a\|\,\|b\|
    \]
    for all $a,b\in A$.
\end{enumerate}
\end{defn}

One important example comes from topology. Let $X$ be a locally compact space, and let $C_0(X)$ denote the algebra of continuous functions
\[
f:X\to \mathbb C
\]
that vanish at infinity. Equipped with the supremum norm
\[
\|f\|=\sup_{x\in X}|f(x)|,
\]
this becomes a Banach algebra.
More generally, if $A$ is a Banach algebra and $X$ is a locally compact space, we may form the algebra
\[
A(X)
\]
of continuous functions
\[
f:X\to A
\]
vanishing at infinity. This construction allows topology to enter directly into $K$-theory.


We are now ready to define higher $K$-groups.

\begin{defn}
Let $A$ be a Banach algebra. For $n\ge 0$, define
\[
K_n(A):=K_0(A(\mathbb R^n)).
\]
These groups are called the \emph{higher $K$-groups} of $A$.
\end{defn}

When $n=0$, this recovers the original Grothendieck group:
\[
K_0(A)=K_0(A(\mathbb R^0)).
\]

Thus higher $K$-groups arise from repeatedly adjoining geometric directions. The appearance of Euclidean space in the definition reflects the fundamentally topological nature of the theory.

These groups extend Theorem \ref{thm: exactness of K0} into a full long exact sequence.

\begin{theorem}
The higher $K$-groups satisfy the following properties.
\begin{enumerate}[label=$(\arabic*)$]
    \item If
    \[
    0\longrightarrow A\longrightarrow B\longrightarrow C\longrightarrow 0
    \]
    is an exact sequence of Banach algebras, then there is a long exact sequence
    \[
    \cdots
    \to
    K_{n+1}(C)
    \to
    K_n(A)
    \to
    K_n(B)
    \to
    K_n(C)
    \to
    K_{n-1}(A)
    \to
    \cdots.
    \]

    \item The inclusion
    \[
    A\hookrightarrow A([0,1])
    \]
    induces an isomorphism
    \[
    K_n(A)\cong K_n(A([0,1])).
    \]
\end{enumerate}
\end{theorem}

The first statement shows that higher $K$-groups behave analogously to derived functors in homological algebra. The second says that $K$-theory is invariant under homotopy, emphasizing once again that the theory is deeply topological.


The connection with topology becomes even more striking through Bott periodicity.

Recall that in the previous section we encountered periodicity phenomena in the homotopy groups of classical groups. Higher $K$-theory turns out to encode precisely this structure.

In fact, for $n\ge 1$, we can identify
\begin{equation}\label{eq: higher ktheory and homotopy}
K_n(A) \cong \pi_{n-1}(\mathrm{GL}(A)), 
\end{equation}
where $\mathrm{GL}(A)$ denotes the stabilized homotopy group $GL_r(A)$ as $r\to \infty$.

Thus higher $K$-groups are fundamentally homotopical invariants.

This leads directly to the following theorem.

\begin{theorem}
Let $A$ be a complex Banach algebra. Then
\[
K_n(A)\cong K_{n+2}(A).
\]

If $A$ is a real Banach algebra, then
\[
K_n(A)\cong K_{n+8}(A).
\]
\end{theorem}

Due to \eqref{eq: higher ktheory and homotopy}, this result can be re-formulated as
\[
\pi_{n-1}(\mathrm{GL}(A))
\cong
\pi_{n+1}(\mathrm{GL}(A))
\]
in the complex case. Observe that this matches the form of Bott periodicity stated in \cref{thm: bott}. Thus, one of the deepest results in homotopy theory becomes encoded in the structure of higher $K$-groups.

\begin{rem}
    This theorem lets us formally define $K_n(A)$ for arbitrary negative integers. For example, in the complex case, we can define $K_n(A)$ as $K_{n+2r}(A)$ for $r$ sufficiently large. In particular, an exact sequence
    \[0 \longrightarrow A \longrightarrow B\longrightarrow C\longrightarrow 0\]
    of Banach algebras like above gives us an exact \textit{hexagon}:
    \[\begin{tikzcd}
 K_0(A) \arrow[r]& K_0(B) \arrow[r]  & K_0(C)\arrow[d]\\
 K_1(A) \arrow[u]& K_1(B)   \arrow[l] & K_1(C). \arrow[l]
    \end{tikzcd}\]
    Since every $K$-group is isomorphic to either $K_0$ or $K_1$, this answers the question of extending the original exact sequence in \cref{thm: exactness of K0}, in both the left and right directions.
\end{rem}


\medskip
The development of higher $K$-theory dramatically broadened Grothendieck's original construction. The Grothendieck group $K_0$ records additive relations among objects, but the higher groups capture hidden homotopical and geometric structure invisible at the level of classes alone.

Today, higher $K$-theory appears throughout modern mathematics:
\begin{itemize}
    \item in topology through generalized cohomology theories,
    \item in algebraic geometry through vector bundles and motivic cohomology,
    \item in number theory through regulators and special values of $L$-functions,
    \item in operator algebras through index theory and noncommutative geometry.
\end{itemize}

What began as a simple process of forming formal differences of objects ultimately led to one of the central organizing ideas of modern mathematics.

\section{$\lambda$-Rings and Adams Operations}

The Grothendieck group $K_0(\mathcal C)$ of an abelian category is designed to capture the additive structure coming from short exact sequences. As we saw in \S 4 and \S 5, many Grothendieck groups naturally carry additional structure. In favorable situations, a tensor product operation on the category induces a multiplication on the Grothendieck group, turning it into a ring. For example, the Grothendieck group $K_0(R)$ of finitely generated projective $R$-modules becomes a ring via the tensor product.

However, tensor products are only one part of the rich algebraic structure present in linear algebra and representation theory. There are many other natural operations on vector spaces and modules, most notably the \emph{exterior powers}
\[
\Lambda^k(V)
\]
and the \emph{symmetric powers}
\[
\Sym^k(V).
\]

These constructions satisfy remarkable algebraic identities. For instance, exterior powers behave with respect to direct sums according to the formula
\begin{equation}\label{eq: exterior}
    \Lambda^k(V\oplus W)
\cong
\bigoplus_{i=0}^k
\Lambda^i(V)\otimes \Lambda^{k-i}(W).
\end{equation}
This identity resembles the binomial theorem and suggests that exterior powers should be regarded as higher-order operations compatible with addition.

The theory of $\lambda$-rings formalizes this idea. Roughly speaking, a $\lambda$-ring is a ring equipped with operations that behave like exterior powers. These structures arise naturally in Grothendieck rings, representation rings, and topological $K$-theory. One of their most important consequences is the construction of the \emph{Adams operations}, which play a central role in algebraic topology and homotopy theory.

For the following discussion, we follow \cite{weibel2013kbook}. We begin with the definition.
\begin{defn}
A commutative ring $K$, together with a family of set maps
\[
\lambda^k:K\to K
\qquad (k\ge0),
\]
is called a \emph{$\lambda$-ring} if the following conditions hold:
\begin{enumerate}[label=$(\arabic*)$]
    \item $\lambda^0(x)=1$ and $\lambda^1(x)=x$ for all $x\in K$;
    
    \item For every $x,y\in K$ and every $k\ge0$,
    \[
    \lambda^k(x+y)
    =
    \sum_{i=0}^k
    \lambda^i(x)\lambda^{k-i}(y).
    \]
\end{enumerate}
\end{defn}
Observe that the relation (2) precisely matches the formula \eqref{eq: exterior}.

Before passing to Grothendieck groups, it is convenient to formulate the same idea at the level of semirings.

\begin{defn}
A commutative semiring $M$, together with maps
\[
\lambda^k:M\to M,
\]
is called a \emph{$\lambda$-semiring} if
\[
\lambda^0(x)=1,
\qquad
\lambda^1(x)=x,
\]
and
\[
\lambda^k(x+y)
=
\sum_{i=0}^k
\lambda^i(x)\lambda^{k-i}(y)
\]
for all $x,y\in M$.
\end{defn}

The semiring $\mathrm{Proj}(R)$ of finitely generated projective $R$-modules is therefore a $\lambda$-semiring, with the operations given by
\[
\lambda^k(P)=\Lambda^k(P).
\]


It is often convenient to package the operations $\lambda^k$ into a single generating function.

\begin{lemma}\label{lem: genfunc}
A commutative ring $K$ equipped with operations $\lambda^k:K\to K$ is a $\lambda$-ring if and only if the map
\[
\lambda_t:K\to 1+tK[[t]]
\]
defined by
\[
\lambda_t(x)
=
\sum_{k\ge0}\lambda^k(x)t^k
\]
is a homomorphism from the additive group of $K$ to the multiplicative group of $1+tK[[t]]$.
\end{lemma}

\begin{proof}
The defining identity for a $\lambda$-ring is equivalent to
\[
\lambda_t(x+y)
=
\lambda_t(x)\lambda_t(y).
\]
Indeed,
\begin{align*}
\lambda_t(x)\lambda_t(y)
&=
\left(
\sum_{k\ge0}\lambda^k(x)t^k
\right)
\left(
\sum_{\ell\ge0}\lambda^\ell(y)t^\ell
\right)\\
&=
\sum_{k\ge0}
\left(
\sum_{i=0}^k
\lambda^i(x)\lambda^{k-i}(y)
\right)t^k.
\end{align*}
Comparing coefficients of $t^k$ yields the result.
\end{proof}

The generating function $\lambda_t(x)$ behaves much like an exponential map:
\[
\lambda_t(x+y)=\lambda_t(x)\lambda_t(y).
\]
This viewpoint is extremely useful in the theory of symmetric functions and algebraic topology.

We now explain why Grothendieck rings naturally carry $\lambda$-ring structures.

\begin{prop}
The Grothendieck ring $K_0(R)$ is a $\lambda$-ring, with operations induced by exterior powers:
\[
\lambda^k([P])=[\Lambda^k(P)].
\]
\end{prop}

\begin{proof}
The semiring $\mathrm{Proj}(R)$ is a $\lambda$-semiring by $(*)$. Hence the map
\[
\lambda_t:\mathrm{Proj}(R)\to 1+tK_0(R)[[t]]
\]
defined by
\[
\lambda_t(P)
=
\sum_{k\ge0}
[\Lambda^k(P)]t^k
\]
is a monoid homomorphism.
Since $K_0(R)$ is the group completion of $\mathrm{Proj}(R)$, the universal property of group completion implies that $\lambda_t$ extends uniquely to a map
\[
\lambda_t:K_0(R)\to 1+tK_0(R)[[t]]
\]
which is additive-to-multiplicative. By Lemma \ref{lem: genfunc}, this endows $K_0(R)$ with a $\lambda$-ring structure.
\end{proof}

An entirely analogous argument gives the following result.

\begin{prop}
Let $G$ be a finite group. Then the representation ring
\[
R(G)
\]
is a $\lambda$-ring, where
\[
\lambda^k([V])=[\Lambda^k(V)]
\]
for finite-dimensional complex representations $V$ of $G$.
\end{prop}

\medskip

We now arrive at one of the most important constructions associated to a $\lambda$-ring: the \emph{Adams operations}. These are certain canonical operations
\[
\psi^k:K\to K
\]
built from the $\lambda$-operations. They behave much better multiplicatively than the operations $\lambda^k$, and they are among the most powerful tools in topological $K$-theory.

Before defining them, we introduce the \emph{augmentation map}
\[
\varepsilon:K\to \mathbb Z,
\]
which plays the role of a dimension or rank function.

For $K_0(R)$, the augmentation map is the rank map
\[
\varepsilon:K_0(R)\to [\mathrm{Spec}(R),\mathbb Z],
\]
while for the representation ring $R(G)$ it is simply the dimension map
\[
\varepsilon:R(G)\to \mathbb Z.
\]

The Adams operations are then defined recursively.

\begin{defn}
The \emph{Adams operations}
\[
\psi^k:K\to K
\]
are defined recursively by
\[
\psi^0(x)=\varepsilon(x),
\qquad
\psi^1(x)=x,
\]
and
\[
\psi^k(x)
=
\lambda^1(x)\psi^{k-1}(x)
-
\lambda^2(x)\psi^{k-2}(x)
+\cdots
+(-1)^{k-1}k\lambda^k(x).
\]
Equivalently,
\[
\psi^k
=
\lambda^1\psi^{k-1}
-
\lambda^2\psi^{k-2}
+\cdots
+
(-1)^{k-1}k\lambda^k.
\]
\end{defn}

These formulas have the same form as     the classical Newton identities relating elementary symmetric polynomials and power-sum symmetric polynomials. In this analogy, $\lambda^k$
plays the role of the elementary symmetric functions, while $\psi^k$ corresponds to the power sums.

The Adams operations are important because, unlike the operations $\lambda^k$, they are ring homomorphisms.

\begin{theorem}
For $K_0(R)$ and $R(G)$, the Adams operations
\[
\psi^k:K\to K
\]
are ring homomorphisms. Moreover,
\[
\psi^j\circ \psi^k
=
\psi^{jk}.
\]
\end{theorem}

In fact, this theorem holds for a large class of sufficiently well-behaved $\lambda$-rings.

\medskip

The Adams operations have had profound applications in algebraic topology. Most famously, they played a decisive role in the solution of the Hopf invariant one problem and in the study of vector fields on spheres. The key idea is that the Adams operations detect subtle information in topological $K$-theory that ordinary cohomology cannot see.

In particular, by studying the action of the Adams operations on the $K$-theory of spheres, one obtains strong constraints on the topology of manifolds and vector bundles. This illustrates one of the central themes of modern mathematics: algebraic operations on Grothendieck groups can encode deep geometric and topological information.

\section{Conclusion}

The theory of Grothendieck groups begins with a modest observation: many mathematical objects admit a natural operation of addition, even when their full classification is hopelessly complicated. Direct sums of modules, vector bundles, sheaves, or representations endow collections of isomorphism classes with the structure of a commutative monoid, and Grothendieck's construction provides a universal method for adjoining formal inverses. What initially appears to be a simple algebraic device ultimately reveals itself as a unifying principle across large portions of modern mathematics.

The progression of ideas throughout this article reflects a recurring mathematical phenomenon. Classical classification theorems suggest that one should seek complete invariants for algebraic objects, yet outside a small number of rigid settings, such ambitions quickly fail. Modules over arbitrary rings, vector bundles over topological spaces, or coherent sheaves on algebraic varieties form categories too vast and intricate to admit satisfactory classifications. Grothendieck groups arise precisely at this point of failure. Rather than attempting to distinguish objects individually, they record additive information while systematically discarding intractable extension-theoretic complexity.

This philosophy is already visible in the construction of
\[
K_0(R),
\]
where finitely generated projective modules are replaced by formal additive combinations modulo direct-sum relations. In favorable situations this process recovers familiar invariants: dimension over a field, rank over a principal ideal domain, or rank together with ideal-class information over a Dedekind domain. More generally, the Grothendieck group of an abelian category retains only the additive shadow of its objects, collapsing short exact sequences into linear relations. The Jordan--H\"older theorem then explains why, in finite-length categories, Grothendieck groups reduce objects to their composition multiplicities.

What is perhaps most striking is the persistence of the same formal mechanism in seemingly unrelated areas of mathematics. In representation theory, the representation ring
\[
R(G)=K_0(\mathrm{Rep}(G))
\]
encodes representations through additive and multiplicative operations arising from direct sum and tensor product. In topology, the Grothendieck completion of vector bundles produces topological $K$-theory, whose structure is governed by Bott periodicity and whose applications include the Atiyah--Singer index theorem. In algebraic geometry, Grothendieck groups interact with characteristic classes and intersection theory through the Grothendieck--Riemann--Roch theorem. In homological algebra, triangulated and derived categories admit Grothendieck groups in which Euler characteristics become manifestations of additivity relations.

The ubiquity of these constructions suggests that Grothendieck groups should not be viewed merely as bookkeeping devices. They represent a systematic method for extracting linear information from nonlinear mathematical worlds. Indeed, the modern language of categorification and decategorification places Grothendieck groups at the boundary between algebraic structures and the richer categories from which they arise. Passing to $K_0$ forgets morphisms, extensions, and higher homological data, but preserves enough structure to reveal deep and often unexpected invariants.

At the same time, Grothendieck groups are only the beginning of a much larger theory. The passage from $K_0$ to higher algebraic $K$-theory reflects the realization that additive information alone does not capture the full complexity of projective modules, vector bundles, or exact categories. Quillen's higher $K$-groups, defined using homotopical and categorical methods, reveal that Grothendieck's original construction is merely the first layer of a profound and far-reaching structure. Modern algebraic $K$-theory now lies at the intersection of topology, arithmetic geometry, homotopy theory, and derived algebraic geometry, with deep connections to regulators, motivic cohomology, and special values of zeta functions.

Nevertheless, the essential idea remains remarkably simple. Whenever a category possesses a meaningful notion of addition, one may attempt to linearize it through a Grothendieck group. That this elementary idea repeatedly leads to deep structural insights is one of the most compelling features of the subject. Grothendieck groups occupy a unique position in mathematics: they are simultaneously universal constructions, computable invariants, and shadows of richer homotopical phenomena yet to emerge.

\vfill 



\end{document}